\documentclass[12pt,reqno]{amsart}
\usepackage{amssymb,amsmath,amsthm}
\usepackage{graphicx}
\usepackage{subfigure}
\usepackage{color}
\usepackage{array}
\usepackage{ulem}
\usepackage[mathlines]{lineno}
\usepackage{tabularx}

\newtheorem{theorem}{Theorem}[section]
\newtheorem{lemma}{Lemma}[section]

\theoremstyle{definition}

\theoremstyle{remark}

\makeatletter
\@namedef{subjclassname@2020}{2020 Mathematics Subject Classification}
\makeatother

\numberwithin{equation}{section}


\date{\today}

\begin{document}

\title[Sequential dichotomies and  uniformities]
{Sequential dichotomies and  uniformities for\\ a  neutral equation}

\maketitle

\centerline{\scshape Shuang Chen}
\medskip
{\footnotesize
\centerline{School of Mathematics and Statistics, Central China Normal University}
 \centerline{ Wuhan, Hubei 430079, China}
   \centerline{{\rm{Email}: schen@ccnu.edu.cn}}
}

\medskip

\centerline{\scshape Weinian Zhang }

{\footnotesize
 \centerline{School of Mathematics, Sichuan University}
   \centerline{Chengdu, Sichuan 610064, China}
    \centerline{{\rm{Email}: matzwn@163.com}}
}

\medskip

\begin{abstract}
Sequential dichotomies of general delay equations are not uniform, which was proved two decades ago.
This however reminds whether the countably infinite many dichotomies of a neutral equation have the sequential uniformity.
In this paper,
considering a scalar neutral equation, we give a negative answer
and prove that the series of the projections of dichotomies is divergent.

\end{abstract}

\parskip 0.3cm

\section{Introduction and main results}

Exponential dichotomy describes the hyperbolicity  of dynamical systems generated
by homogeneous linear differential equations or maps.
It formulates the decomposition of the phase space associated with a dynamical system
into a (weak) stable invariant subspace and a (weak) unstable one (\cite{Chen-Shen-20,Coppel-78,Zhang-93}).
More precisely,
on a Banach space $\mathcal{B}$,
the $C^{0}$ semigroup $\{T(t):t\geq 0\}$ is said to admit an {\it exponential dichotomy}
(see \cite{Henry-81}) on an interval $\mathbb{R}^{+}:=[0,+\infty)$
if there exist projections  $\Pi_{-}$ and $\Pi_{+}:=I-\Pi_{-}$ on $\mathcal{B}$, where $I$ is the identity,
and real constants $\alpha$, $\beta$ with $\alpha>\beta$ and $K\geq 1$  such that
\begin{enumerate}
\item[{\bf (D1)}] $T(t)\Pi_{\pm}=\Pi_{\pm}T(t)$ for $t\geq 0$;
\vskip 3pt

\item[{\bf (D2)}] $T(t)|_{\mathcal{R}(\Pi_{+})}$ is an isomorphism from the range $\mathcal{R}(\Pi_{+})$ of $\Pi_{+}$ onto itself,
which defines the inverse of $T(t)|_{\mathcal{R}(\Pi_{+})}$, denoted by $T(-t):\mathcal{R}(\Pi_{+}) \to \mathcal{R}(\Pi_{+})$;
\vskip 3pt

\item[{\bf (D3)}] $\|T(t)\Pi_{-}\|\leq Ke^{\beta t}$ for $t\geq 0$,
and $\|T(-t)\Pi_{+}\|\leq Ke^{-\alpha t}$ for $t\geq 0$.
\end{enumerate}
We call $\alpha$, $\beta$ and $K$ an {\it upper exponent}, a {\it lower exponent}, and a {\it bound} of this dichotomy, respectively.
For notational convenience,
we use $\mathcal{E}(\beta, \alpha; K)$
to denote this dichotomy with exponents $\alpha,\beta$ and bound $K$,
and refer to $\Pi_{+}$ and $\Pi_{-}$ as the {\it unstable (dichotomy) projection} and the {\it stable (dichotomy) projection} respectively.

Unlike any linear autonomous system of ODEs (abbr. ordinary differential equations),
which only admits finitely many exponential dichotomies since each matrix has finitely many eigenvalues,
an infinite-dimensional linear evolutionary system
has a generic property that
the corresponding semigroup, generated by the solution operator,
can admit infinitely many exponential dichotomies.
For example, the scalar delay equation
\begin{eqnarray}\label{RDDE}
\dot x(t)= x(t-1),\ \ \ \ x\in\mathbb{R},
\end{eqnarray}
has the characteristic equation $\lambda-e^{-\lambda}=0$, from which we can compute its eigenvalues (see \cite{JKHale-Verduyn})
$$
\lambda_n=-\ln(2n\pi+\frac{\pi}{2})+ {\bf i}\left(2n\pi+\frac{\pi}{2}-\frac{\ln(2n\pi+\frac{\pi}{2})}{2n\pi+\frac{\pi}{2}}\right)+\Re(n)
$$
and their conjugacies $\bar\lambda_n$, $n=1,2,...$,
where ${\bf i}=\sqrt{-1}$ and $\Re(n)$ is the remainder with
the real parts and the imaginary parts
$$
{\rm Re}\Re(n)=O((\frac{\ln(2n\pi+\frac{\pi}{2})}{2n\pi+\frac{\pi}{2}})^2),
~~~
{\rm Im}\Re(n)=O((\frac{\ln(2n\pi+\frac{\pi}{2})}{2n\pi+\frac{\pi}{2}})^3).
$$
One can check that ${\rm Re}\lambda_{n+1}<{\rm Re}\lambda_{n}$ for sufficiently large $n$, implying that
equation \eqref{RDDE} has a sequence of spectral gaps $({\rm Re}\lambda_{n+1},{\rm Re}\lambda_{n})$,
an open interval of $\mathbb{R}$ covering no real parts of the eigenvalues of Eq.\eqref{RDDE},
and therefore a sequence of dichotomies $\mathcal{E}(\beta_{n}, \alpha_{n}; K_{n})$ with projections $P_{n}^{\pm}$,
where the pairs of exponents $\alpha_{n}$ and $\beta_{n}$ satisfy
${\rm Re}\lambda_{n}>\alpha_{n}>\beta_{n}>{\rm Re}\lambda_{n+1}$ for all sufficiently large $n>0$,
each bound $K_{n}$ can be expressed in terms of $\alpha_{n}$ and $\beta_{n}$ as shown in \cite[Lemma 6.2, p.213]{JKHale-Verduyn},
and each $P_{n}^{+}$  is the sum of spectral projections of the eigenvalues satisfying ${\rm Re}\lambda>\alpha_{n}$.
We say that equation (\ref{RDDE}) has {\it sequential dichotomies}.

For the sequential dichotomies, there are two basic questions:
\begin{enumerate}
\item[{\bf (Q1)}]
For every $\phi\in C[-1,0]$,
the Banach space of continuous functions on $[-1,0]$
endowed with the supremum norm $\|\cdot\|_{\infty}$,
do the dichotomy projections $P_{n}^{+}$ satisfy
$\|P_{n}^{+}\phi-\phi\|_{\infty}\to 0$ as $n\to +\infty$?

\vskip 3pt
\item[{\bf (Q2)}]
Do the sequential dichotomies have a uniform bound?

\end{enumerate}
Question {\bf (Q1)} naturally arises from the approximation of solutions by finite-dimensional projections.
This topic has been investigated by Bellman and Cooke \cite{Bellemmaan-Cooke-63}, Banks and  Manitius \cite{Banks-Manitus},
and Hale and Verduyn Lunel \cite{JKHale-Verduyn}.
It was indicated in \cite[Corollary 3.12, p.25]{Verduyn-01} that
the linear space spanned by all eigenfunctions of the generator for Eq. \eqref{RDDE} is dense in $C[-1,0]$
and each solution of Eq. \eqref{RDDE} has a convergent series expansion for $t>0$.
Question {\bf (Q2)} requires a uniform constant $K:=\sup_{i\in\mathbb{N}}K_{i}<+\infty$, which
was encountered earlier in 1975 by Banks and Manitius \cite{Banks-Manitus} when
they wanted to know the growth of the sequence of bounds just below (2.29) of \cite{Banks-Manitus}.
It was settled down by Hale and Zhang \cite{Hale-Zhang-04} in 2004,
who proved that the sequential dichotomies do not have the uniformity for the general delay equation
\begin{eqnarray}\label{RFDE}
\dot x(t)= Lx_t,
\end{eqnarray}
where $L: C([-r,0],\mathbb{R}^n) \to \mathbb{R}^n$ is a bounded linear operator,
$r>0$ is a constant and the notation $x_t(\theta):= x(t+\theta)$ is defined for all $\theta \in [-r,0]$.
As we know in \cite{BDC-18,Perron-30,Zhou-Zhang-16},
uniformity is also considered for a single dichotomy because of the dependence on the initial time or the angle between two eigenspaces,
even when we consider nonautonomous systems in a finite-dimensional space.
In order to distinguish the different sense of uniformity, we call the version considered in this paper the {\it sequential uniformity}.
The sequential uniformity was denied in \cite{Hale-Zhang-04} for autonomous delay equations, but one may ask:
{\it Does a neutral equation also not have the sequential uniformity?
}

In this paper,  we investigate the sequential uniformity for the scalar neutral equation
\begin{eqnarray}\label{eq:NE}
x'(t)+cx'(t-1)+ax(t)+bx(t-1)=0, \ \ \ \ \  x\in\mathbb{R},
\end{eqnarray}
where $a$, $b$ and $c\neq 0$ are real constant.
All eigenvalues of Eq. \eqref{eq:NE} are determined by the roots of  its characteristic equation
which lie in a strip $\nu_1<{\rm Re}\lambda<\nu_2$ for some real constants $\nu_{1}$ and $\nu_{2}$
(see Lemma \ref{lm:eigen-distribution} below).
This shows that
it admits finitely many spectral gaps or countably infinitely many spectral gaps
and therefore Eq. \eqref{eq:NE} admits finite many or countably infinitely many exponential dichotomies, as
stated below.

\begin{theorem}
\label{thm:1}
Eq. \eqref{eq:NE} with $c\neq 0$ admits countably infinite many exponential dichotomies if
\begin{eqnarray}\label{abc}
a\neq b/|c|\ \ \ \mbox{ and }\ \ \  \ln |c|\neq -\frac{1}{2}(a+b/|c|).
\end{eqnarray}
Otherwise,
Eq. \eqref{eq:NE} with $c\neq 0$ admits finite many exponential dichotomies.
\end{theorem}
\noindent
Theorem \ref{thm:1} will be proved in section \ref{sec:proof-exist}
by applying the distribution of eigenvalues
obtained in section \ref{sec:roots}.
In the case of countably infinite many exponential dichotomies,
we further answer the two aforementioned questions {\bf (Q1)} and {\bf (Q2)},
i.e., the sequential uniformity and the convergence problem for projections series.

\begin{theorem} \label{thm:2}
Suppose that $c\ne 0$ and (\ref{abc}) is true.
Let $\mathcal{E}(\beta_{m}, \alpha_{m}; K_{m}), m\in\mathbb{N}$, denote all exponential dichotomies of Eq. \eqref{eq:NE}
as obtained in Theorem \ref{thm:1}
and let $P_{m}^{\pm}$, $m\in\mathbb{N}$, denote their corresponding projections.
Then the following statements are true:
\begin{enumerate}
\item[(i)] The norms $\|P_{m}^{\pm}\|$ of $P_{m}^{\pm}: C[-1,0]\to C[-1,0]$ satisfy $\sup_{m\in\mathbb{N}}\|P_{m}^{\pm}\|=+\infty$.
\vskip 3 pt

\item[(ii)] Sequential dichotomies of Eq. \eqref{eq:NE} are not uniform.
\end{enumerate}
\end{theorem}
\noindent
Theorem \ref{thm:2}, to be proved in section \ref{sec:nonuniformity},
indicates that there exists a function $\phi\in C[-1,0]$ such that the
sequence $\{\|P_{m}^{\pm}\phi-\phi\|_{\infty}\}_{m\in\mathbb{N}}$ is divergent.
Due to different spectral structures,
Hale and Zhang's method \cite{Hale-Zhang-04}
for Eq. \eqref{RFDE}
is not applicable to Eq. \eqref{eq:NE}.
We develop a new approach by applying the properties of orthonormal systems (see \cite{Olevskii}),
which are abstracted from the eigenfunctions of an auxiliary equation $x'(t)+cx'(t-1)=0$.


\section{Roots of the characteristic equation}
\label{sec:roots}

Substituting $x(t)=e^{\lambda t}$ in Eq. \eqref{eq:NE} yields its characteristic equation
\begin{eqnarray}\label{eq:character}
h(\lambda):=\lambda+ c \lambda e^{-\lambda}+a+b e^{-\lambda} =0.
\end{eqnarray}
The distribution of the roots of Eq. \eqref{eq:character}
plays an important role in the study of sequential dichotomies and  their sequential uniformity.
We first present some basic results on the asymptotic properties of the roots
and also refer to \cite{Bellemmaan-Cooke-63} for more information on the roots of Eq. \eqref{eq:character}.

\begin{lemma}\label{lm:eigen-distribution}
The following statements hold:
\begin{enumerate}
\item[(i)]
There exist two real constants $\nu_{1}$ and $\nu_{2}$ such that all complex roots of Eq. \eqref{eq:character}
lie in the vertical strip $\nu_1<{\rm Re}\lambda<\nu_2$ of the complex plane $\mathbb{C}$.
\vskip 3pt

\item[(ii)]
Let  $z_{n}:=\ln |c|+{\bf i}(2n\pi+\arg (-c))$ for $n\in\mathbb{Z}$
and $B_{n}(\varepsilon)$ denote small open balls centered at $z_{n}$ with radius $\varepsilon$.
Then for a sufficiently small $\varepsilon>0$,
there exists an integer $n_{\varepsilon}>0$ such that
all roots of Eq. \eqref{eq:character} with $|\lambda|>n_{\varepsilon}$ lie in
$\cup_{|n|>n_{\epsilon}}B_{n}(\varepsilon)$
but only one complex root $\lambda_{n}$ is covered by $B_{n}(\varepsilon)$ ($|n|>n_{\varepsilon}$) exactly for each $n$.

\vskip 3pt

\item[(iii)]
The real part ${\rm Re}\lambda_n$ and the imaginary part $ {\rm Im}\lambda_n$ of $\lambda_n$  satisfy
\begin{eqnarray}
 {\rm Re}\lambda_n \!\!\! &=& \!\!\! \ln |c|+\Re_{1}(n),
\label{Real-part-1}
\\
 {\rm Im}\lambda_n \!\!\! &=& \!\!\! (2n\pi+\arg (-c))+\Re_{2}(n),
\label{Imaginary-part-1}
\end{eqnarray}
where
$\Re_1(n)=O(|\frac{1}{n}|), \ \Re_2(n)=O(|\frac{1}{n}|)
$
for sufficiently large $|n|$.
\vskip 3pt

\item[(iv)] Eq. \eqref{eq:character} has at most three different real roots.
\end{enumerate}
\end{lemma}
\begin{proof}
In the case $\lambda\neq -a$, Eq. \eqref{eq:character} is equivalent to the form
\begin{eqnarray}\label{h-equiv}
e^{\lambda}=-\frac{c\lambda+b}{\lambda+a}=-c+\frac{ac-b}{\lambda+a}.
\end{eqnarray}
Since
\[
|e^{\lambda}|\to 0, \ \ \ |\frac{ac-b}{\lambda+a}|\to 0,\ \ \ \mbox{ as }\ {\rm Re}\lambda\to -\infty,
\]
there exists a constant $\nu_1<0$ such that for all $\lambda$ with ${\rm Re}\lambda\leq \nu_{1}$,
\[
|e^{\lambda}|+|\frac{ac-b}{\lambda+a}|\leq \frac{|c|}{2}<|c|.
\]
This together with \eqref{h-equiv} implies that all complex roots of Eq. \eqref{eq:character} lie in ${\rm Re}\lambda>\nu_1$.
Similarly, we have that
\[
|e^{\lambda}|\to +\infty, \ \ \ |\frac{ac-b}{\lambda+a}|\to 0,\ \ \ \mbox{ as }\ {\rm Re}\lambda\to +\infty.
\]
These limits imply that there exists a constant $\nu_{2}>0$ such that for all $\lambda$ with ${\rm Re}\lambda\geq \nu_{2}$,
\[
|e^{\lambda}|\geq 2|c|>|c|+|\frac{ac-b}{\lambda+a}|.
\]
This together with \eqref{h-equiv} implies that all complex roots of Eq. \eqref{eq:character} lie in ${\rm Re}\lambda<\nu_{2}$.
Hence we conclude that all complex roots of Eq. \eqref{eq:character} lie in $\nu_{1}<{\rm Re}\lambda<\nu_{2}$.
This proves  (i).

In the case $\lambda\neq 0$, Eq. \eqref{eq:character} is equivalent to the form
\[
e^{\lambda}+c+\frac{ae^{\lambda}+b}{\lambda}=0.
\]
Note that  $e^{\lambda}+c=0$ has countably infinite many complex roots $z_{n}$, $n\in\mathbb{Z}$.
Then (ii) is implied by Rouch\'e Theorem (\cite[Theorem 10.43, p.225]{Rudin-87}).

In order to obtain the formulae in (iii), substituting
\[
\lambda_{n}=\ln |c|+\Re_{1}(n)+{\bf i} ((2n\pi+\arg (-c))+\Re_{2}(n))
\]
in Eq. \eqref{eq:character}, we obtain
\begin{eqnarray}\label{eq:expans}
e^{\Re_{1}(n)++{\bf i}\Re_{2}(n)}-1=\frac{ac-b}{c}\cdot\frac{1}{1+a\lambda_{n}^{-1}}\cdot\frac{\bar\lambda_{n}}{|\lambda_{n}|^{2}}
\end{eqnarray}
for sufficiently large $|n|$,
where $\bar{\lambda}$ denotes the conjugacy of $\lambda$.
By (ii) of this lemma,
\[
|\frac{ac-b}{c}\cdot\frac{1}{1+a\lambda_{n}^{-1}}\cdot\frac{\bar\lambda_{n}}{|\lambda_{n}|^{2}}|=O(|\frac{1}{n}|)
\]
for sufficiently large $|n|$.
Then $|\Re_{1}(n)++{\bf i}\Re_{2}(n)|\to 0$ as $|n|\to +\infty$.
Expanding the left-hand side of Eq. \eqref{eq:expans} in a power series about $\Re_{1}(n)++{\bf i}\Re_{2}(n)$,
we further see that   $\Re_1(n)=O(|\frac{1}{n}|)$ and $\Re_2(n)=O(|\frac{1}{n}|)$ for sufficiently large $|n|$.
This proves (iii).

A direct computation yields that
\[
h''(\lambda)=(c\lambda+b-2c)e^{-\lambda},
\]
which has exactly one real zero at $\lambda={(2c-b)}/{c}$.
This shows that $h'(\lambda)$ has at most two different real zeros
and then $h'(\lambda)$ has at most three different real zeros
by the Mean-Value Theorem, implying (iv).
The proof is finished.
\end{proof}

Define two sets $\Sigma$ and $\Omega$ associated with the zeros of $h$ as follows
\begin{eqnarray}\label{df:Sigma-Omega}
\Sigma:= {\rm the\ closure\ of}\ \left\{{\rm Re}\lambda:\,h(\lambda)=0\right\},
\ \ \ \ \
\Omega:=\mathbb{R}/\Sigma.
\end{eqnarray}
We will see that the existence of sequential dichotomies for Eq. \eqref{eq:NE}
is determined by the configuration of $\Omega$ (equivalently, $\Sigma$).
Thus, we call $\Sigma$ and $\Omega$  the {\it dichotomous spectrum set} and the {\it dichotomous resolvent set}
of Eq. \eqref{eq:NE}, respectively.

In order to study the configuration of $\Omega$,
we present a preparatory lemma.
Note that  $\ln |c|$ is the unique  accumulation point of $\Sigma$ (see Lemma \ref{lm:eigen-distribution}).
We first normalize $c$ in $h(\lambda)$, i.e., changing it to $\pm 1$.
Substituting $\lambda=z+\ln|c|$ in $h(\lambda)=0$ yields
\begin{eqnarray}\label{eq:App-1}
z(1+ \delta e^{-z})+A+B e^{-z} =0,
\end{eqnarray}
where
$A=a+\ln |c|,\  B=\ln |c|+{b}/{|c|}$,
and
\begin{eqnarray}
\delta
=\left\{
\begin{aligned}
1,   &&\ \ \ \ \ \mbox{ if }\ c>0,\\
-1, &&\ \ \ \ \ \mbox{ if }\ c<0.
\end{aligned}
\right.
\end{eqnarray}
Concerning the roots of Eq. \eqref{eq:App-1}, we have the next lemma.

\begin{lemma}\cite[Theorems 2 and 3]{Junca-Lombard-14}
\label{lm:App-1}
For $\delta=\pm 1$, the properties of the roots $z_{n}$, $n\in\mathbb{Z}$, of Eq. \eqref{eq:App-1}
are given in Tables \ref{tab-1} and \ref{tab-2} respectively.
\begin{table}[!htbp]
\centering
\begin{tabular}{c|c}
  \hline
{\rm Parameter region} & {\rm The distribution of eigenvalues} \\
\hline
$A>|B|$  & ${\rm Re}z_{n}<0$, \ $n\in\mathbb{Z}$ \\
\hline
$|A|<-B$ & ${\rm Re}z_{n}>0$, \ $n\in\mathbb{Z}$ \\
\hline
$|A|<B$  & ${\rm Re}z_{n}>0$ \mbox{for} $|n|\geq 1$, $z_{0}<0$ \\
\hline
$A<-|B|$  &  ${\rm Re}z_{n}<0$ \mbox{for} $|n|\geq 1$, $z_{0}>0$ \\
\hline
$A=|B|>0$  & ${\rm Re}z_{n}=0$ \mbox{for} $|n|\geq 1$, $z_{0}=-A<0$ \mbox{if} $B>0$ \\
\hline
$A=-B\leq0$  &  ${\rm Re}z_{n}=0$, \ $n\in\mathbb{Z}$  \\
\hline
$A=B<0$  & ${\rm Re}z_{n}=0$ \mbox{for} $|n|\geq 1$, $z_{0}=-A>0$ \\
\hline
\end{tabular}
\vskip 0.4cm
\caption{The distribution of eigenvalues of Eq. \eqref{eq:App-1} with $\delta=1$.}
\label{tab-1}
\end{table}

\begin{table}[!htbp]
\centering
\begin{tabular}{c|c}
  \hline
{\rm Parameter region} & {\rm The distribution of eigenvalues}\\
\hline
$A>|B|$   & ${\rm Re}z_{n}<0$,\ $n\in\mathbb{Z}$\\
\hline
$|A|<B$  & ${\rm Re}z_{n}>0$, \ $n\in\mathbb{Z}$\\
\hline
$|A|<-B$  & ${\rm Re}z_{n}>0$ \mbox{for} $|n|\geq 1$, $z_{0}<0$\\
\hline
$A<-|B|$   & ${\rm Re}z_{n}<0$ \mbox{for} $|n|\geq 1$, $z_{0}>0$\\
\hline
$A=|B|>0$  & ${\rm Re}z_{n}=0$ \mbox{for} $|n|\geq 1$, $z_{0}=-A<0$ \mbox{if} $B<0$\\
\hline
$A=-B<0$  & ${\rm Re}z_{n}=0$ \mbox{for} $|n|\geq 1$, $z_{0}=-A>0$ \\
\hline
$A=B\geq 0$  & ${\rm Re}z_{n}=0$, \ $n\in\mathbb{Z}$\\
\hline
\end{tabular}
\vskip 0.4cm
\caption{The distribution of eigenvalues of Eq. \eqref{eq:App-1} with $\delta=-1$.}
\label{tab-2}
\end{table}
\end{lemma}

Now we characterize the configuration of $\Omega$ in the following two lemmas.

\begin{lemma}\label{lm:conf-c-positve}
Suppose that $c>0$ in Eq. \eqref{eq:character}. Then the following assertions are true:
\begin{enumerate}
\item[(i)] If $a+\ln c>|\ln c+b/c|$,
then there is a strictly increasing sequence $\{\gamma_{k}\}_{k\geq 0}$ in $\mathbb{R}$ with the limit $\lim_{k\to +\infty}\gamma_{k}=\ln c$
such that
\begin{eqnarray*}
\Omega=(-\infty,\gamma_{0})\cup\left(\bigcup_{k=0}^{\infty}(\gamma_{k},\gamma_{k+1})\right)\cup (\ln c, +\infty).
\end{eqnarray*}

\item[(ii)] If $|a+\ln c|<-(\ln c+b/c)$,
then there is a strictly decreasing sequence $\{\gamma_{k}\}_{k\geq 0}$ in $\mathbb{R}$ with the limit $\lim_{k\to +\infty}\gamma_{k}=\ln c$
such that
\begin{eqnarray*}
\Omega=(-\infty,\ln c)\cup \left(\bigcup_{k=0}^{\infty}(\gamma_{k+1},\gamma_{k})\right)\cup (\gamma_{0}, +\infty).
\end{eqnarray*}

\item[(iii)] If $|a+\ln c|<\ln c+b/c$,
then there is a strictly decreasing sequence $\{\gamma_{k}\}_{k\geq 1}$ in $\mathbb{R}$ with the limit $\lim_{k\to +\infty}\gamma_{k}=\ln c$
and a real constant $\gamma_{0}<\ln c$ such that
\begin{eqnarray*}
\Omega=(-\infty,\gamma_{0})\cup(\gamma_{0},\ln c)\cup\left(\bigcup_{k=1}^{\infty}(\gamma_{k+1},\gamma_{k})\right)\cup (\gamma_{1}, +\infty).
\end{eqnarray*}

\item[(iv)] If $a+\ln c<-|\ln c+b/c|$,
then there is a strictly increasing sequence $\{\gamma_{k}\}_{k\geq 1}$ in $\mathbb{R}$ with the limit $\lim_{k\to +\infty}\gamma_{k}=\ln c$
and a real constant $\gamma_{0}>\ln c$ such that
\begin{eqnarray*}
\Omega=(-\infty,\gamma_{1})\cup\left(\bigcup_{k=1}^{\infty}(\gamma_{k},\gamma_{k+1})\right)\cup (\ln c,\gamma_{0})\cup (\gamma_{0}, +\infty).
\end{eqnarray*}

\item[(v)] If $|a+\ln c|=|\ln c+b/c|$,
then $\Omega$ contains at most three connected intervals.
\end{enumerate}
\end{lemma}

\begin{proof}
If $c>0$, the parameters $A$, $B$ and $\delta$ in Eq. \eqref{eq:App-1} satisfy
\[
A=a+\ln c,\ \ \  B=\ln c+b/c, \ \ \ \delta=1.
\]
Let $z_{n}$, $n\in\mathbb{Z}$, be defined as in Table \ref{tab-1} of Lemma \ref{lm:App-1}.
Note that  $\lambda$ is a root of $h(\lambda)=0$ if and only if $z=\lambda-\ln c$ is a root of Eq. \eqref{eq:App-1}.
Then $\Omega$, defined in \eqref{df:Sigma-Omega}, can be presented as
\begin{eqnarray}\label{omg-1}
\Omega=\mathbb{R}/(\{{\rm Re} z_{n}+\ln c :\, n\in \mathbb{Z}\}\cup \{\ln c\}).
\end{eqnarray}

To prove statements (i)-(iv), we choose $\gamma_{k}$, $k\geq 0$ such that
$
\{\gamma_{k}\}_{k\geq 0}=\{{\rm Re} z_{n}+\ln c\}_{n\in \mathbb{Z}},
$
and arrange $\gamma_{k}$  in the following order:
If  $a+\ln c>|\ln c+b/c|$ or $|a+\ln c|<-(\ln c+b/c)$, which implies that $A>|B|$ or $|A|<-B$,
then $\gamma_{k}$, $k\geq 0$, are ordered according to decreasing $|\gamma_{k}-\ln c|$;
If $|a+\ln c|<\ln c+b/c$ or $a+\ln c<-|\ln c+b/c|$,
which implies that $|A|<B$ or $A<-|B|$,
then we let $\gamma_{0}=z_{0}$ and order $\gamma_{k}$, $k\geq 1$, according to decreasing $|\gamma_{k}-\ln c|$.
Thus, by Table \ref{tab-1} of Lemma \ref{lm:App-1} and \eqref{omg-1}, we conclude statements (i)-(iv).

In the case that $|a+\ln c|=|\ln c+b/c|$, we have $|A|=|B|$.
Using Table \ref{tab-1} of Lemma \ref{lm:App-1} and \eqref{omg-1} again, we conclude statement (v).
Therefore, the proof is completed.
\end{proof}

\begin{lemma}\label{lm:conf-c-negative}
Suppose that $c<0$ in Eq. \eqref{eq:character}.
Then the following assertions are true:
\begin{enumerate}
\item[(i)] If $a+\ln (-c)>|\ln (-c)-b/c|$,
then there is a strictly increasing sequence $\{\gamma_{k}\}_{k\geq 0}$ in $\mathbb{R}$ with the limit $\lim_{k\to +\infty}\gamma_{k}=\ln (-c)$ such that
\begin{eqnarray*}
\Omega=(-\infty,\gamma_{0})\cup\left(\bigcup_{k=0}^{\infty}(\gamma_{k},\gamma_{k+1})\right)\cup (\ln (-c), +\infty).
\end{eqnarray*}

\item[(ii)]
If $|a+\ln (-c)|<\ln (-c)-b/c$,
then there is a strictly decreasing sequence $\{\gamma_{k}\}_{k\geq 0}$ in $\mathbb{R}$ with the limit $\lim_{k\to +\infty}\gamma_{k}=\ln (-c)$
such that
\begin{eqnarray*}
\Omega=(-\infty,\ln (-c))\cup \left(\bigcup_{k=0}^{\infty}(\gamma_{k+1},\gamma_{k})\right)\cup (\gamma_{0}, +\infty).
\end{eqnarray*}

\item[(iii)]
 If $|a+\ln (-c)|<-(\ln (-c)-b/c)$,
then there is a strictly decreasing sequence $\{\gamma_{k}\}_{k\geq 1}$ in $\mathbb{R}$ with the limit $\lim_{k\to +\infty}\gamma_{k}=\ln (-c)$
and a real constant $\gamma_{0}<\ln (-c)$ such that
\begin{eqnarray*}
\Omega=(-\infty,\gamma_{0})\cup(\gamma_{0},\ln (-c))\cup\left(\bigcup_{k=1}^{\infty}(\gamma_{k+1},\gamma_{k})\right)\cup (\gamma_{1}, +\infty).
\end{eqnarray*}

\item[(iv)] If $a+\ln (-c)<-|\ln (-c)-b/c|$,
then there is a strictly increasing sequence $\{\gamma_{k}\}_{k\geq 1}$ in $\mathbb{R}$ with the limit $\lim_{k\to +\infty}\gamma_{k}=\ln (-c)$
and a real constant $\gamma_{0}>\ln c$ such that
\begin{eqnarray*}
\Omega=(-\infty,\gamma_{1})\cup\left(\bigcup_{k=1}^{\infty}(\gamma_{k},\gamma_{k+1})\right)\cup (\ln (-c),\gamma_{0})\cup (\gamma_{0}, +\infty).
\end{eqnarray*}

\item[(v)] If $|a+\ln (-c)|=|\ln (-c)-b/c|$,
then $\Omega$ contains at most three connected intervals.
\end{enumerate}
\end{lemma}

\begin{proof}
If $c<0$, the parameters $A$, $B$ and $\delta$ in Eq. \eqref{eq:App-1} satisfy
\[
A=a+\ln (-c),\ \ \  B=\ln (-c)-b/c, \ \ \ \delta=-1.
\]
Let $z_{n}$, $n\in\mathbb{Z}$, be defined as in Table \ref{tab-2} of Lemma \ref{lm:App-1}.
Note that  $\lambda$ is a root of $h(\lambda)=0$ if and only if $z=\lambda-\ln (-c)$ is a root of Eq. \eqref{eq:App-1}.
Then $\Omega$, defined in \eqref{df:Sigma-Omega}, can be presented as
\begin{eqnarray}\label{omg-2}
\Omega=\mathbb{R}/(\{{\rm Re} z_{n}+\ln (-c) :\, n\in \mathbb{Z}\}\cup \{\ln (-c)\}).
\end{eqnarray}

To prove statements (i)-(iv), we choose $\gamma_{k}$, $k\geq 0$, such that
$
\{\gamma_{k}\}_{k\geq 0}=\{{\rm Re} z_{n}+\ln (-c)\}_{n\in \mathbb{Z}},
$
and arrange $\gamma_{k}$  in the following order:
If  $a+\ln (-c)>|\ln (-c)-b/c|$ or $|a+\ln (-c)|<\ln (-c)-b/c$, which implies that $A>|B|$ or $|A|<B$,
then $\gamma_{k}$, $k\geq 0$, are ordered according to decreasing $|\gamma_{k}-\ln c|$;
If $|a+\ln (-c)|<-(\ln (-c)-b/c)$ or $a+\ln (-c)<-|\ln (-c)-b/c|$,
which implies that $|A|<-B$ or $A<-|B|$,
then we let $\gamma_{0}=z_{0}$ and order $\gamma_{k}$, $k\geq 1$, according to decreasing $|\gamma_{k}-\ln c|$.
Thus, by Table \ref{tab-2} of Lemma \ref{lm:App-1} and \eqref{omg-2}, we conclude statements (i)-(iv).

In the case that $|a+\ln (-c)|=|\ln (-c)-b/c|$, we have $|A|=|B|$.
Using Table \ref{tab-2} of Lemma \ref{lm:App-1} and \eqref{omg-2} again, we conclude statement (v).
Therefore, the proof is completed.
\end{proof}


\section{Existence of dichotomies for the solution semigroup}
\label{sec:proof-exist}

In this section,
consider the sequential dichotomies for Eq. \eqref{eq:character} on the phase space $C[-1,0]$,
i.e., the Banach space of continuous functions on $[-1,0]$ endowed with the supremum norm $\|\cdot\|_{\infty}$.
By Theorem 8.3 of \cite[p. 65]{JKHale-Verduyn},
for a given initial value $\phi\in C[-1,0]$,
there exists a unique solution $x(t,\phi)$ for $t\geq -1$ of Eq. \eqref{eq:NE} satisfying
the initial conditions $x(\theta,\phi)=\phi(\theta)$ for $\theta\in[-1,0]$.
Then we can define the solution operator $T(t):  C[-1,0]\to  C[-1,0]$ by
\begin{eqnarray*}
T(t)\phi=x_{t}(\cdot,\phi),\ \ \ \ \ t\geq 0,\ \ \phi\in  C[-1,0],
\end{eqnarray*}
where  $x_{t}(\theta,\phi)=x(t+\theta,\phi)$ for $\theta\in [-1,0]$.
Furthermore, the family $\{T(t):\,t\geq 0\}$ is a strongly continuous semigroup on $ C[-1,0]$ (see \cite{Pazy-83}),
i.e.,
\[
\lim_{t\to 0^{+}}\|T(t)\phi-\phi\|_{\infty}=0,\ \ \ \ \ \forall\, \phi\in C[-1,0].
\]
The corresponding infinitesimal generator $\mathcal{A}$ is given by
\begin{eqnarray}\label{df:generator}
\mathcal{A}\phi:=\lim_{t\rightarrow 0^{+}}\frac{T(t)\phi-\phi}{t},\ \ \ \ \
\phi\in \mathcal{D}(\mathcal{A}),
\end{eqnarray}
where the domain $\mathcal{D}(\mathcal{A})$ of  $\mathcal{A}$ is determined by
\begin{eqnarray*}
\mathcal{D}(\mathcal{A})=\left\{\phi\in  C[-1,0]: \dot\phi\in  C[-1,0],\, \dot\phi(0)+c\dot\phi(-1)+a\phi(0)+b\phi(-1)=0 \right\}.
\end{eqnarray*}
Here $\dot\phi(0)$ and $\dot\phi(-1)$ denote the left-hand derivative at $0$ and the right-hand derivative at $-1$, respectively.
We refer to Chapter 9 of \cite{JKHale-Verduyn} for more detail on neutral equations and their solution  semigroups.

In order to analyze the spectrum of the generator $\mathcal{A}$,
we extend the real Banach space $ C[-1,0]$ to $ \mathcal{C}_{\mathbb{C}}:= C[-1,0]\oplus {\bf i} C[-1,0]$,
and $\mathcal{A}$ on $\mathcal{C}_{\mathbb{C}}$  (see \cite{Kato-95}) such that
\[
\mathcal{A}(\phi+{\bf i}\psi)=\mathcal{A}\phi+{\bf i}\mathcal{A}\psi,
\ \ \ \ \ \forall\,
\phi,\,\psi\in \mathcal{D}(\mathcal{A}).
\]
For convenience,
we still write $\mathcal{A}$ for its extension on $ \mathcal{C}_{\mathbb{C}}$.

The basic properties of the generator $\mathcal{A}$ are summarized in the next lemma.

\begin{lemma}\label{lm:A-spectrum}
The generator $\mathcal{A}$ has the following properties:
\begin{enumerate}
\item[(i)]
The spectrum $\sigma (\mathcal{A})$ contains eigenvalues (point spectrum) only,
which are the complex roots of the characteristic equation \eqref{eq:character}.
An eigenvector associated with an eigenvalue $\lambda$ is $e^{\lambda\theta}$, defined for all $\theta\in [-1,0]$.
\vskip 3pt

\item[(ii)] For each $\lambda$ in the resolvent set $\rho(\mathcal{A})$ of the generator $\mathcal{A}$,
the resolvent $R(\lambda,\mathcal{A})$ is given by
\begin{eqnarray*}
R(\lambda,\mathcal{A})\varphi(\theta)
= e^{\lambda\theta}\!\left\{(h(\lambda))^{-1}F(\lambda,\varphi)+\int_{\theta}^{0}e^{-\lambda s} \varphi(s)d s\right\},
\ \ \ \ \ \varphi\in \mathcal{C}_{\mathbb{C}},
\end{eqnarray*}
where
\begin{eqnarray*}
F(\lambda,\varphi):=
\varphi(0)+ c\varphi(-1)-c\lambda e^{-\lambda}\int^{0}_{-1}e^{-\lambda s}\varphi(s)d s
 -be^{-\lambda}\int^{0}_{-1}e^{-\lambda s}\varphi(s)d s.
\end{eqnarray*}
\vskip 3pt

\item[(iii)] If $\lambda$ a simple eigenvalue of $\mathcal{A}$,
then its spectral projection
\begin{eqnarray}\label{prj}
P_{\lambda} \varphi := \frac{1}{2\pi {\bf i}}\int_{\Gamma_\lambda} R(\omega,\mathcal{A})\varphi\, d \omega,
\ \ \ \ \ \varphi \in  \mathcal{C}_{\mathbb{C}},
\end{eqnarray}
where $\Gamma_\lambda$ is a contour such that $\lambda$ is the unique eigenvalue inside, can be expressed as
\begin{eqnarray*}
(P_\lambda\varphi)(\theta)  = e^{\lambda \theta}(h'(\lambda))^{-1}F(\lambda,\varphi), \ \ \ \ \ \varphi \in  \mathcal{C}_{\mathbb{C}}.
\end{eqnarray*}
\vskip 3pt

\item[(iv)] If $\lambda$ is an eigenvalue of $\mathcal{A}$,
then $\bar\lambda$ is also an eigenvalue of $\mathcal{A}$ and
\[
P_{\lambda}+P_{\bar\lambda}\,:\, C[-1,0]\to C[-1,0].
\]
\end{enumerate}
\end{lemma}
\begin{proof}
It follows from \cite[Lemma 2.1, p.263]{JKHale-Verduyn} that
$\sigma (\mathcal{A})$ only contains eigenvalues, which are the roots of the characteristic equation $h(\lambda)=0$.
For each eigenvalue $\lambda$, which solves $h(\lambda)=0$,
we can compute that function $e^{\lambda\theta}$ for $\theta\in [-1,0]$ lies in $\mathcal{D}(\mathcal{A})$
and $\mathcal{A}e^{\lambda\,\theta}=\lambda e^{\lambda\,\theta}$, $\theta\in [-r,0]$,
implying that $e^{\lambda\,\cdot}$ is an eigenvector of $\lambda$.
This proves (i).

Suppose that $\lambda\in\rho(\mathcal{A})$.
For each $\varphi \in  \mathcal{C}_{\mathbb{C}}$, the equation $(\lambda I-\mathcal{A})\phi=\varphi$
has a unique solution $\phi \in \mathcal{D}(\mathcal{A})$,
i.e., $\phi$ solves
\begin{eqnarray}\label{boundary}
\left\{
\begin{aligned}
 &\lambda\phi-\dot \phi= \varphi,\\
 &\dot\phi(0)+c\dot\phi(-1)+a\phi(0)+b\phi(-1)=0.
\end{aligned}
\right.
\end{eqnarray}
By the Variation of Constant Formula,
the first equation in \eqref{boundary} has  solutions of the form
\begin{eqnarray}\label{solution}
\phi(\theta)=e^{\lambda\theta}\phi(0)+ \int^{0}_{\theta}e^{\lambda(\theta-s)} \varphi(s)d s, \quad -1\leq\theta\leq 0,
\end{eqnarray}
where $\phi(0)$ will be determined by the second equation of \eqref{boundary}.
Substituting \eqref{solution} in the second equation of \eqref{boundary} yields
\begin{eqnarray*}
0\!\!\!&=&\!\!\!
\lambda \phi(0)-\varphi(0)+c \lambda e^{-\lambda}\phi(0)-c\varphi(-1)
+c\lambda e^{-\lambda}\int^{0}_{-1}e^{-\lambda s}\varphi(s)d s\\
\!\!\!&&\!\!\!
+ a\phi(0)+b e^{-\lambda}\phi(0)+be^{-\lambda}\int^{0}_{-1}e^{-\lambda s}\varphi(s)d s,
\end{eqnarray*}
Note that $h(\lambda)\neq 0$.
Then $\phi(0)$ is uniquely determined by
\[
\phi(0)=(h(\lambda))^{-1}F(\lambda,\varphi).
\]
Substituting it in \eqref{solution} yields (ii).

Suppose that  $\lambda$ is a simple eigenvalue of $\mathcal{A}$.
Then $\lambda$ is also a simple root of \eqref{eq:character} (see \cite{JKHale-Verduyn}).
By (ii) in this lemma,
\begin{eqnarray*}
(P_\lambda\varphi)(\theta) = I_1(\theta)+ I_2(\theta),\ \ \ \ \ \varphi\in  \mathcal{C}_{\mathbb{C}}, \ \theta\in[-1,0],
\end{eqnarray*}
where
\begin{eqnarray*}
I_1(\theta)\!\!\!&:=&\!\!\!
      \frac{1}{2\pi {\bf i}}\int_{\Gamma_\lambda} e^{\omega\theta}(h(\omega))^{-1}F(\omega,\varphi)\,dz,
\\
I_2(\theta)\!\!\!&:=&\!\!\!
      \frac{1}{2\pi {\bf i}}\int_{\Gamma_\lambda} e^{\omega\theta}\int_{\theta}^{0}e^{-\omega s} \varphi(s)ds \, d \omega.
\end{eqnarray*}
Note that  $e^{\theta\,\cdot}\int_{\theta}^{0}e^{- s\,\cdot} \varphi(s)ds$
and $e^{\theta\,\cdot}F(\cdot\,,\varphi)$ are analytic in $\mathbb{C}$ for each $\theta\in [-1,0]$.
Then by the Residue Theorem (see \cite[Theorem 10.42, p.224]{Rudin-87}), we get that
\begin{eqnarray*}
I_1(\theta)\!\!\!&=&\!\!\!\frac{1}{2\pi {\bf i}}\int_{\Gamma_\lambda} (h(\omega))^{-1}F(\omega,\varphi)e^{\omega\theta} d \omega
\nonumber\\
\!\!\!&=&\!\!\!\lim_{\omega\rightarrow \lambda} (\omega-\lambda) (h(\omega))^{-1}F(\omega,\varphi)e^{\omega\theta}\nonumber\\
\!\!\!&=&\!\!\!e^{\lambda\theta}(h'(\lambda))^{-1}F(\lambda,\varphi).
\end{eqnarray*}
and $I_2(\theta)\equiv 0$. This yields (iii).

Suppose that  $\lambda$ is an eigenvalue of $\mathcal{A}$.
Statement (i) in this lemma yields $h(\lambda)=0$. Then we get $h(\bar\lambda)= \overline{h(\lambda)}=0$.
Using (i) again, we have that $\bar\lambda$ is also an eigenvalue of $\mathcal{A}$.
Let $\Gamma_\lambda$ be a contour such that $\lambda$ is the unique eigenvalue inside.
Then $\bar \Gamma_\lambda=\{z\in \mathbb{C}\,:\, \bar z\in \Gamma_{\lambda}\}$ is a contour such that $\bar\lambda$ is the unique eigenvalue inside.

For each $\phi\in \mathcal{C}_{\mathbb{C}}$ and $w\in \rho(\mathcal{A})$,
there exists a function $\varphi\in \mathcal{C}_{\mathbb{C}}$ such that $R(\omega,\mathcal{A})\phi=\varphi$.
This yields $(\omega I-\mathcal{A})\varphi=\phi$. Since
$$
\phi=\overline{\phi}
=\overline{(\omega I-\mathcal{A})\varphi}
=\overline{\omega\varphi}-\overline{\mathcal{A}\varphi}
=(\overline{\omega} I-\mathcal{A})\overline{\varphi},
$$
we have
\begin{eqnarray*}
R(\overline{\omega},\mathcal{A})\phi=\overline{\varphi}=\overline{R(\omega,\mathcal{A})\phi}.
\end{eqnarray*}
This yields
\[
P_{\bar \lambda} \phi= \frac{1}{2\pi {\bf i}}\int_{\bar \Gamma_\lambda} R(\omega,\mathcal{A})\phi\, d \omega
   =\overline{\frac{1}{2\pi {\bf i}}\int_{\Gamma_\lambda} R(\omega,\mathcal{A})\phi\, d \omega}=\overline{P_{\lambda} \phi},
\]
which implies  (iv). Therefore the proof is now completed.
\end{proof}

Recall that the dichotomous spectrum set $\Sigma$ and the dichotomous resolvent set $\Omega$  are defined by \eqref{df:Sigma-Omega}.
In the following, we prove Theorem \ref{thm:1}.
\\

{\bf Proof of Theorem \ref{thm:1}.}
We first prove that for each $\mu\in \Omega$ (see \eqref{df:Sigma-Omega}),
there exist real constants $\alpha_{\mu}$, $\beta_{\mu}$ with $\beta_{\mu}<\mu<\alpha_{\mu}$
and $K_{\mu}\geq 1$ such that the solution semigroup $\{T(t)\,:\,t\geq0\}$ on $\mathcal{C}_{\mathbb{C}}$ admits an exponential dichotomy
$\mathcal{E}(\beta_{\mu}, \alpha_{\mu}; K_{\mu})$.
Fix a constant $\mu\in \Omega$. Substituting  $x(t)=y(t)e^{\mu t}$ in Eq. \eqref{eq:NE} yields
\begin{eqnarray}\label{eq:auxillary}
y'(t)+ce^{-\mu}y'(t-1)+(a+\mu)y(t)+(b+c\mu)e^{-\mu}y(t-1)=0,
\end{eqnarray}
which has the characteristic equation $h(\lambda+\mu)=0$, i.e.,
\begin{eqnarray*}
\lambda(1+ce^{-(\lambda+\mu)})+(a+\mu)+(b+c\mu)e^{-\lambda+\mu}=0.
\end{eqnarray*}
Recall that $\Omega$ is defined in  \eqref{df:Sigma-Omega} and $\ln |c|\notin \Omega$.
Then there exists a small $\varepsilon(\mu)>0$ such that
\begin{eqnarray}\label{ineq:ED-cond-1}
1+ce^{-(\lambda+\mu)}\neq 0,\ \ \ h(\lambda+\mu)\neq 0
\end{eqnarray}
for $-\varepsilon(\mu)<{\rm Re}\lambda<\varepsilon(\mu)$.
By  Theorem 4.2 in \cite[p.278]{JKHale-Verduyn},
Eq. \eqref{eq:auxillary} admits an exponential dichotomy
$\mathcal{E}(-\varepsilon(\mu)/2,\varepsilon(\mu)/2;K_{\mu})$ for some constant $K_{\mu}\geq 1$.
This implies that Eq. \eqref{eq:NE} admits an exponential dichotomy $\mathcal{E}(\mu-\varepsilon(\mu)/2,\mu+\varepsilon(\mu)/2;K_{\mu})$.

Next, we prove that for each $\mu\in \Omega$
the solution semigroup $\{T(t)\,:\,t\geq0\}$ on $C[-1,0]$ admits an exponential dichotomy.
For the dichotomy $\mathcal{E}(\mu-\varepsilon(\mu)/2,\mu+\varepsilon(\mu)/2;K_{\mu})$,
let $P^{+}_{\mu}$ and $P^{-}_{\mu}$ denote the corresponding unstable and stable projections.
Recall that $\ln |c|$ is the unique accumulation point of  $\Sigma$.
Then Eq. \eqref{eq:character} has finitely many roots $\lambda$s such that either ${\rm Re}\lambda\leq \mu$ or ${\rm Re}\lambda\geq \mu$.
Without loss of generality,
assume that  Eq. \eqref{eq:character} has finitely many roots $\lambda=\omega_{1}$, $\omega_{2}$, ..., $\omega_{n_{\nu}}$
such that ${\rm Re}\lambda\geq \mu$.
It follows from \cite[Lemma 2.1, p.263]{JKHale-Verduyn} that
\begin{eqnarray}\label{eq:proj}
P^{+}_{\mu}\phi =\sum_{i=1}^{n_{\nu}}P_{\omega_{i}} \phi,\ \ \ \ \ \forall\, \phi \in \mathcal{C}_{\mathbb{C}},
\end{eqnarray}
where $P_{\omega_{i}}$ is the spectral projection of the eigenvalue $\omega_{i}$ for each $i$ (see \eqref{prj}).
This implies that different connected intervals of $\Omega$ determine different dichotomies.
Using \eqref{eq:proj} and (iv) of Lemma \ref{lm:A-spectrum}, we further have
\[
P^{+}_{\mu}\,:\, C[-1,0] \to C[-1,0].
\]
Thus $\{T(t)\,:\,t\geq0\}$ on the space $C[-1,0]$ admits an exponential dichotomy $\mathcal{E}(\mu-\varepsilon(\mu)/2,\mu+\varepsilon(\mu)/2;K_{\mu})$.

Finally, we count exponential dichotomies of Eq. \eqref{eq:NE}.
If \eqref{abc} holds, then $|a+\ln |c||\neq |\ln |c|+b/|c||$.
Using Lemmas \ref{lm:conf-c-positve} and \ref{lm:conf-c-negative} yields
that $\Omega$ contains countably infinite many connected intervals.
Then Eq. \eqref{eq:NE} admits countably infinite many exponential dichotomies.
Otherwise, if $|a+\ln |c||=|\ln |c|+b/|c||$,
then by (v) of Lemma \ref{lm:conf-c-positve} and (v) of Lemma \ref{lm:conf-c-negative},
we have that $\Omega$ contains at most three connected intervals.
Then by the preceding discussion, Eq. \eqref{eq:NE} admits at most three exponential dichotomies.
Therefore, the proof is completed.
\hfill $\Box$


\section{Non-uniformity of sequential dichotomies for the solution semigroup}
\label{sec:nonuniformity}

In this section,
we consider the uniformity of sequential dichotomies.
When Eq. \eqref{eq:NE} admits countably infinite many exponential dichotomies,
we will prove that  this sequence of dichotomies has no uniform bounds, i.e.,
it is sequentially non-uniform.

\subsection{An auxiliary equation}\

We first consider an auxiliary equation
\begin{eqnarray}\label{eq:NE-auxi}
x'(t)+cx'(t-1)=0, \ \ \ \ \  x\in\mathbb{R},
\end{eqnarray}
which has the characteristic equation
\begin{eqnarray}\label{eq:h-0}
h_{0}(\lambda):=\lambda(1+ ce^{-\lambda})=0.
\end{eqnarray}
Let $\mathcal{B}$ denote the infinitesimal generator of the solution semigroup of Eq. \eqref{eq:NE-auxi}.
A direct computation yields that all eigenvalues of $\mathcal{B}$ are given by
\begin{eqnarray*}
z^0_0:=0, \ \ \ z_{n}:=\ln |c|+{\bf i}(2n\pi+\arg (-c)), \ \ \ n\in \mathbb{Z}.
\end{eqnarray*}
If $c=-1$, then $z^0_0=0$ is the unique root of multiplicity 2.
Otherwise, all roots of Eq. \eqref{eq:h-0} are simple.

The eigenvalues $z_{n}$ ($n\in\mathbb{Z}$) of $\mathcal{B}$ are closely related to an orthogonal system in $L^{2}[-1,0]$,
i.e.,
the set of square-integrable functions on $[-1,0]$ with the $L^{2}$-norm.
We also refer to \cite{Olevskii} for more information on orthogonal systems.
In fact, we have the following results.

\begin{lemma}\label{lm:orth}
Define $\{\chi_{n}\}_{n\in\mathbb{Z}}$ by
\[
\chi_{n}(\theta):=e^{{\bf i}2n\pi\theta},\ \ \ \ \ \theta\in [-1,0],\ \ n\in \mathbb{Z}.
\]
Then $\{\chi_{n}\}_{n\in\mathbb{Z}}$ is an orthogonal system in $L^{2}[-1,0]$,
i.e.,
\begin{eqnarray*}
\int^{0}_{-1} \overline{\chi}_{n}(\theta) \chi_m(\theta)\,d\theta
=\left\{
\begin{aligned}
0\ \ \  \mbox{ if } \ n\neq m,\\
1\ \ \ \mbox{ if } \ n=m,
\end{aligned}
\right.
\end{eqnarray*}
where $\overline{\chi}_{n}(\theta)$ is the conjugacy of $\chi_{n}(\theta)$.
Furthermore,  let
\[
\hat a_{n}(\phi)=\int_{-1}^{0} e^{z_{n}\theta}\phi(\theta)\,d\theta,
\ \ \ \phi \in \mathcal{C}_{\mathbb{C}}.
\]
Then there exists a constant $C_{\star}>0$ such that
\[
\sum_{n=-\infty}^{+\infty}|\hat a_{n}(\phi)|^{2}\leq C_{\star} \|\phi\|_{\infty}^{2}.
\]
\end{lemma}
\begin{proof}
This first statement can be proved by a direct computation (see also \cite{Rudin-87}).
Note that
\[
\hat a_{n}(\phi)=\int_{-1}^{0} e^{z_{n}\theta}\phi(\theta)\,d\theta
                =\int_{-1}^{0} \chi_{n}(\theta) e^{(\ln |c|+{\bf i}\arg (-c))\theta}\phi(\theta)\,d\theta.
\]
Then $\hat a_{n}(\phi)$ are the Fourier coefficients  of function $e^{(\ln |c|+{\bf i}\arg (-c))\theta}\phi(\theta)$.
Since $\phi$ is in $\mathcal{C}_{\mathbb{C}} \subset L^{2}[-1,0]$,
using the Parseval theorem (see \cite[p.91]{Rudin-87}) yields
\[
\sum_{n=-\infty}^{+\infty}|\hat a_{n}(\phi)|^{2}
 =\int_{-1}^{0}|e^{(\ln |c|+{\bf i}\arg (-c))\theta}\phi(\theta)|^{2}\,d\theta
 \leq \|\phi\|_{\infty}^{2}\cdot\max\{1,|c|^{-2}\}.
\]
This completes the proof.
\end{proof}

By Lemma \ref{lm:A-spectrum},
the resolvent $R(\lambda,\mathcal{B})$, $\lambda\in \rho(\mathcal{B})$, is given by
\begin{eqnarray}\label{res:B}
R(\lambda,\mathcal{B})\phi(\theta)
= e^{\lambda\theta}\!\left\{(h_{0}(\lambda))^{-1}F_{0}(\lambda,\phi)+\int_{\theta}^{0}e^{-\lambda s} \phi(s)d s\right\},
\ \ \ \ \ \phi\in  \mathcal{C}_{\mathbb{C}},
\end{eqnarray}
where
\begin{eqnarray*}
F_{0}(\lambda,\phi):=
\phi(0)+ c\phi(-1)-c\lambda e^{-\lambda}\int^{0}_{-1}e^{-\lambda s}\phi(s)d s.
\end{eqnarray*}
Recall that $B_{n}(\varepsilon)$ are small open balls with center at $z_{n}$ and radius $\varepsilon$.
For a sufficiently small $\varepsilon$ and each $|n|>n_{\varepsilon}$,
no eigenvalues of $\mathcal{A}$ are on the boundaries $\partial B_{n}(\varepsilon)$ of $B_{n}(\varepsilon)$
(see (ii) of Lemma \ref{lm:eigen-distribution}).
In the following,
we consider the perturbation of $R(\lambda,\mathcal{A})$ from  $R(\lambda,\mathcal{B})$
along the boundaries $\partial B_{n}(\varepsilon)$ for  $|n|>n_{\varepsilon}$.

\begin{lemma}{\rm (Perturbation of the resolvent)}
\label{pert-resol}
For a sufficiently small $\varepsilon>0$ and  any $\lambda\in \partial B_{n}(\varepsilon)$ with $|n|>n_{\varepsilon}$,
there exists a constant $C_{\star}>0$ and  a sequence of positive constants $c_n$, $n\in \mathbb{Z}$, with
$\sum^{+\infty}_{n=-\infty} c_n^2 \leq C_{\star}$
such that
\begin{eqnarray*}
\|(R(\lambda,\mathcal{A})-R(\lambda,\mathcal{B}))\phi\|_{\infty}\leq C_{\star} (\frac{1}{n^{2}}+\frac{c_{n}}{n})
\end{eqnarray*}
 for  each $\phi\in \mathcal{C}_\mathbb{C}$ with $\|\phi\|_{\infty}\leq 1$.
Particularly, let
$
\mathcal{S}_{\mathbb{C}}=
\left\{\phi \in  \mathcal{C}_\mathbb{C}:{\rm support}(\phi)\subset [-1,0]\right\}.
$
Then
\begin{eqnarray*}
\|(R(\lambda,\mathcal{A})-R(\lambda,\mathcal{B}))\phi\|_{\infty}\leq C_{\star} \frac{c_{n}}{n}
\end{eqnarray*}
for each $\phi\in\mathcal{S}_\mathbb{C}$ with $\|\phi\|_{\infty}\leq 1$.
\end{lemma}
\begin{proof}
Fix $\phi\in \mathcal{C}_\mathbb{C}$ such that $\|\phi\|_{\infty}\leq 1$
and $\lambda\in \partial B_{n}(\varepsilon)$ with $|n|>n_{\varepsilon}$.
By Lemmas \ref{lm:eigen-distribution} and \ref{lm:A-spectrum},
we have   $\lambda\in \rho(\mathcal{A})\cap\rho(\mathcal{B})$.
A direct computation yields
\begin{eqnarray}\label{resolvminus}
\begin{aligned}
&(R(\lambda,\mathcal{A})-R(\lambda,\mathcal{B}))\phi(\theta) \\
&= e^{\lambda\theta}\left(h^{-1}(\lambda)F(\lambda,\phi)- h^{-1}_{0}(\lambda)F_{0}(\lambda,\phi)\right)\\
&= e^{\lambda\theta}\left( h^{-1}_{0}(\lambda)(F(\lambda,\phi)-F_{0}(\lambda,\phi))+(h^{-1}(\lambda)- h^{-1}_{0}(\lambda))F(\lambda,\phi)\right).
\end{aligned}
\end{eqnarray}

To prove this lemma,
we first estimate $|h^{-1}_{0}(\lambda)|$ and $|h^{-1}(\lambda)- h^{-1}_{0}(\lambda)|$.
For each $\lambda\in \partial B_{n}(\varepsilon)$,
we can write $\lambda$ as
\[
\lambda=\ln |c|+{\bf i}(2n\pi+\arg (-c))+\varepsilon e^{{\bf i}\zeta}
\]
for some $\zeta\in [0,2\pi)$. This together with \eqref{eq:h-0} yields
$
h^{-1}_{0}(\lambda)=\lambda^{-1}(1-e^{-\varepsilon e^{{\bf i}\zeta}}).
$
It follows that there exists a small  $\varepsilon$ such that
\begin{eqnarray}\label{est:h-0}
|h^{-1}_{0}(\lambda)|\leq 3|\lambda|^{-1}.
\end{eqnarray}
Note that
$
h^{-1}(\lambda)=h^{-1}_{0}(\lambda)(1+h^{-1}_{0}(\lambda)(a+be^{-\lambda}))^{-1}.
$
Without loss of generality,
we can choose $n_{\varepsilon}>0$ such that
\[
|h^{-1}(\lambda)|\leq 6|\lambda|^{-1}.
\]
This together with \eqref{est:h-0} yields that
\begin{eqnarray}\label{diff:h-h0}
|h^{-1}(\lambda)-h^{-1}_{0}(\lambda)|=|h^{-1}(\lambda)h^{-1}_{0}(\lambda)(a+be^{-\lambda})|\leq C_{\star} |\lambda|^{-2}
\end{eqnarray}
for some constant $C_{\star}>0$.

Next, we estimate $|F(\lambda,\phi)-F_{0}(\lambda,\phi)|$.
Recall that $F(\lambda,\phi)$ is defined in (ii) of Lemma \ref{lm:A-spectrum}
and $F_{0}(\lambda,\phi)$ is defined below \eqref{res:B}.
We can compute
\[
|F(\lambda,\phi)-F_{0}(\lambda,\phi)|
  =|-be^{-\lambda}\int^{0}_{-1}e^{-\lambda s}\phi(s)d s|
  \leq C_{\star}|a_{n}(\phi)|
\]
for some $C_{\star}>0$, where
\[
a_{n}(\phi)=\int^{0}_{-1}e^{-\lambda s}\phi(s)d s=\int^{0}_{-1}e^{-z_{n} s}\cdot(e^{-\varepsilon e^{{\bf i}\zeta}s}\phi(s))d s.
\]
By Lemma \ref{lm:orth}, there exists a constant $C_{\star}$ such that
\[
\sum^{+\infty}_{n=-\infty} |a_n(\phi)|^2 \leq  C_{\star}.
\]

Finally, we estimate $|F(\lambda,\phi)|$.  Let
\[
G(\lambda,\phi):=
-(c+b\lambda^{-1})e^{-\lambda}\int^{0}_{-1}e^{-\lambda s}\phi(s)d s.
\]
Then we have
\[
F(\lambda,\phi)=\phi(0)+ c\phi(-1)+\lambda G(\lambda,\phi).
\]
Using Lemma \ref{lm:orth} again,
we see that there exists a constant $C_{\star}>0$ such that
\begin{eqnarray}\label{est:F-phi}
|G(\lambda,\phi)|\leq C_{\star}|a_{n}(\phi)|,
\ \ \
|F(\lambda,\phi)|\leq (1+c)+ C_{\star}|a_{n}(\phi)||\lambda|.
\end{eqnarray}
Thus the first part is proved by \eqref{resolvminus}\,-\,\eqref{est:F-phi} and Lemma \ref{lm:eigen-distribution}.

If $\phi\in\mathcal{S}_\mathbb{C}$, then we have
\[
F(\lambda,\phi)=\lambda G(\lambda,\phi).
\]
By the first equality in \eqref{est:F-phi}, we get
\[
|F(\lambda,\phi)|\leq C_{\star}|a_{n}(\phi)||\lambda|.
\]
This together with \eqref{resolvminus}\,-\,\eqref{diff:h-h0} and Lemma \ref{lm:eigen-distribution} yields the second part.
This completes the proof.
\end{proof}

Note that $z_{n}$, $n\in\mathbb{Z}$,
are simple eigenvalues of $\mathcal{B}$,
except for $z_{0}={\bf i}\pi$ when $c=-1$ (see \cite[Corollary 2.1, p.264]{JKHale-Verduyn}).
Then, by (iii) of Lemma \ref{lm:A-spectrum}, the corresponding spectral projections $Q_{z_{n}}$ ($|n|\geq1$) are of the form
\begin{equation}\label{df:Q-zn}
\begin{split}
Q_{z_{n}}\phi &:=\frac{1}{2\pi {\bf i}}\int_{\Gamma_{z_{n}}} R(\lambda,\mathcal{B})\phi\, d \lambda\\
              &=  \frac{1}{z_{n}}\left(\phi(0)+ c\phi(-1)+z_{n}\int^{0}_{-1}e^{-z_{n} s}\phi(s)d s\right)e^{z_{n}\theta},
              \ \ \ \phi\in  \mathcal{C}_{\mathbb{C}},
\end{split}
\end{equation}
where  each $\Gamma_{z_{n}}$ is a contour such that $z_{n}$ is the unique eigenvalue inside.
In the following,
we further consider the perturbation of $P_{\lambda_{n}}$ from $Q_{z_{n}}$.

\begin{lemma}{\rm (Perturbation of projections)}\label{pert-prjs}
Let $\varepsilon$, $n_{\varepsilon}$, $C_{\star}$, $c_{n}$ be defined as in Lemma \ref{pert-resol}.
Then for $|n|>n_{\varepsilon}$ and $\phi\in \mathcal{C}_\mathbb{C}$ with $\|\phi\|_{\infty}\leq 1$,
the spectral projections $P_{\lambda_{n}}$ and $Q_{z_{n}}$ satisfy
\begin{eqnarray*}
\|P_{\lambda_{n}}\phi-Q_{z_{n}}\phi\|_{\infty} \!\!\!&\leq&\!\!\! \varepsilon  C_{\star} (\frac{1}{n^{2}}+\frac{c_{n}}{n}),
\\
\|\sum_{|n|>n_{\varepsilon}}(P_{\lambda_{n}}\phi-Q_{z_{n}}\phi)\|_{\infty} \!\!\!&\leq&\!\!\!
    \frac{\varepsilon \pi^{2}}{3}C_{\star}+\frac{\varepsilon \pi}{\sqrt{3}}C_{\star}^{\frac{3}{2}}.
\end{eqnarray*}
Particularly, for $|n|>n_{\varepsilon}$ and $\phi\in \mathcal{S}_\mathbb{C}$ with $\|\phi\|_{\infty}\leq 1$,
\begin{eqnarray*}
\|P_{\lambda_{n}}\phi-Q_{z_{n}}\phi\|_{\infty} \!\!\!&\leq&\!\!\! \frac{\varepsilon C_{\star} c_{n}}{n},
\\
\|\sum_{|n|>n_{\varepsilon}}(P_{\lambda_{n}}\phi-Q_{z_{n}}\phi)\|_{\infty} \!\!\!&\leq&\!\!\! \frac{\varepsilon \pi}{\sqrt{3}}C_{\star}^{\frac{3}{2}}.
\end{eqnarray*}
\end{lemma}
\begin{proof}
Fix $\phi\in \mathcal{S}_\mathbb{C}$ with $\|\phi\|_{\infty}\leq 1$.
Recall that for $|n|>n_{\varepsilon}$
the generator  $\mathcal{A}$ (resp. $\mathcal{B}$) has the unique eigenvalue $\lambda_{n}$ (resp. $z_{n}$) in $B_{n}(\varepsilon)$.
Note that
\begin{eqnarray*}
P_{\lambda_{n}}\phi \!\!\!&=&\!\!\! \frac{1}{2\pi {\bf i}}\int_{\partial B_{n}(\varepsilon)} R(\lambda,\mathcal{A})\phi\, d \omega,
\\
Q_{z_{n}}\phi \!\!\!&=&\!\!\! \frac{1}{2\pi {\bf i}}\int_{\partial B_{n}(\varepsilon)} R(\lambda,\mathcal{B})\phi\, d \omega.
\end{eqnarray*}
Then by Lemma \ref{pert-resol} there exists a constant $C_{\star}>0$ such that
\begin{eqnarray*}
\|P_{\lambda_{n}}\phi-Q_{z_{n}}\phi\|_{\infty}
   \!\!\!&=&\!\!\!   \|\frac{1}{2\pi {\bf i}}\int_{\partial B_{n}(\varepsilon)} (R(\lambda,\mathcal{A})-R(\lambda,\mathcal{B}))\phi\, d \omega\|_{\infty}\\
   \!\!\!&\leq &\!\!\!  \varepsilon \sup_{\lambda\in \partial B_{n}(\varepsilon)}\|(R(\lambda,\mathcal{A})-R(\lambda,\mathcal{B}))\phi\|_{\infty}\\
   \!\!\!&\leq &\!\!\!  \varepsilon  C_{\star} (\frac{1}{n^{2}}+\frac{c_{n}}{n}),
\end{eqnarray*}
where $c_{n}$\,s
are defined as in Lemma \ref{pert-resol} and satisfy $\sum^{+\infty}_{n=-\infty} c_n^2 \leq C_{\star}$.
It follows that
\begin{eqnarray*}
\|\sum_{|n|>n_{\varepsilon}}(P_{\lambda_{n}}\phi-Q_{z_{n}}\phi)\|_{\infty}
   \!\!\!&\leq &\!\!\! \varepsilon C_{\star}  \sum_{|n|>n_{\varepsilon}}(\frac{1}{n^{2}}+\frac{c_{n}}{n})\\
   \!\!\!&\leq &\!\!\! \varepsilon C_{\star}\left\{  \sum_{|n|\geq1}\frac{1}{n^{2}}+ \sqrt{\sum^{+\infty}_{n=-\infty} c_n^2}\cdot \sqrt{\sum_{|n|\geq1}\frac{1}{n^{2}}}\right\}\\
   \!\!\!&\leq &\!\!\! \frac{\varepsilon \pi^{2}}{3}C_{\star}+\frac{\varepsilon \pi}{\sqrt{3}}C_{\star}^{\frac{3}{2}},
\end{eqnarray*}
where we use the Cauchy-Schwarz Inequality and $\sum_{n=1}^{+\infty}1/n^2=\pi^{2}/6$.
This proves the first part.
By Lemma \ref{pert-resol}, we can get the second one similarly.
Therefore the proof is now completed.
\end{proof}

In the end of this subsection,
we consider a convergence problem on the partial sums of the spectral projections $Q_{z}$, $z\in\sigma(\mathcal{B})$.
More precisely,
consider a family of subsets of $\sigma(\mathcal{B})$, denoted by $\{\Omega_{\ell}\}_{\ell=1}^{\infty}$,
which satisfies the following properties:

\begin{enumerate}
\item[(\bf P)]
$\mathcal{N}_{\ell}:=\#\Omega_{\ell}<+\infty$ for $\ell\geq 1$,
$\Omega_{1}\subsetneq \Omega_{2}\subsetneq \cdots \Omega_{\ell}\subsetneq \cdots$,
and $\bigcup_{\ell=1}^{\infty}\Omega_{\ell}\subset\sigma(\mathcal{B})$.
\end{enumerate}
Here we use $\#$ to denote a function counting the elements in a set.
Define the partial sums of the spectral projections $Q_{z}$, $z\in\Omega_{\ell}$, as the form
\begin{eqnarray}\label{df:T-ell}
\mathcal{T}_{\ell}=\sum_{z\in\Omega_{\ell}}Q_{z}.
\end{eqnarray}
Then $\mathcal{T}_{\ell}$, $\ell\geq 1$, define a family of bounded linear operators from $\mathcal{C}_{\mathbb{C}}$ to itself.
The corresponding norms of $\mathcal{T}_{\ell}$ are denoted by $\|\mathcal{T}_{\ell}\|$.
Then we have the following lemma.

\begin{lemma}\label{lm:Fseries-nonuniform}
Suppose that $\mathcal{N}_{\ell}$, the number of elements in $\Omega_{\ell}$,  satisfiess
\begin{eqnarray}\label{eq:lim-N-m}
\lim_{\ell\to +\infty}\frac{\mathcal{N}_{\ell+1}}{\mathcal{N}_{\ell}}=1.
\end{eqnarray}
Then there exist a function $\phi_{*}\in \mathcal{S}_{\mathbb{C}}$
and a positive measure set $\mathcal{U}\subset (-1,0)$  in the sense of Lebesgue measure such that
\[
\lim_{\ell\to +\infty} |\mathcal{T}_{\ell}\phi_{*}(\theta)| =+\infty,
\ \ \ \ \ \forall\,\theta\in\mathcal{U}.
\]
Furthermore, the partial sums $\mathcal{T}_{\ell}$, $\ell\geq 1$, satisfy
\[
\sup_{\ell\geq 1}\|\mathcal{T}_{\ell}\|=+\infty.
\]
\end{lemma}
\begin{proof}
Without loss of generality,
assume that there exists an integer $\ell_{0}$ such that $z^0_0=0\in \Omega_{\ell_{0}}$.
Then all eigenvalues in
\[
\Omega_{\ell}\setminus\Omega_{\ell_{0}}=:\left\{z_{\ell_{k}}\,:\,\mathcal{N}_{\ell_{0}}+1\leq k\leq \mathcal{N}_{\ell} \right\},
\ \ \ \ell>\ell_{0},
\]
are simple ones of the generator $\mathcal{B}$.
It follows from \eqref{df:Q-zn} that for each $\phi\in\mathcal{S}_{\mathbb{C}}$ and $\ell>\ell_{0}$,
\begin{equation}\label{eq:Tm-Tm0}
\begin{split}
(\mathcal{T}_{\ell}-\mathcal{T}_{\ell_{0}})\phi(\theta)
&= \sum_{k=\mathcal{N}_{\ell_{0}}+1}^{\mathcal{N}_{\ell}} Q_{z_{\ell_{k}}}\phi(\theta)\\
&= \sum_{k=\mathcal{N}_{\ell_{0}}+1}^{\mathcal{N}_{\ell}} e^{z_{\ell_{k}}\theta}\int^{0}_{-1}e^{-z_{\ell_{k}} s}\phi(s)d s.\\
&= \sum_{k=\mathcal{N}_{\ell_{0}}+1}^{\mathcal{N}_{\ell}}
   \left(\int^{0}_{-1} \overline{\chi}_{\ell_{k}}(s) \widetilde{\phi}(s) d s \right)\chi_{\ell_{k}}(\theta)   e^{(\ln |c|+{\bf i}\arg (-c))\theta},
\end{split}
\end{equation}
where
\[
\chi_{\ell_{k}}(\theta)=e^{{\bf i}2\ell_{k}\pi\theta},
\ \ \
\widetilde{\phi}(s)=e^{-(\ln |c|+{\bf i}\arg (-c))s}\phi(s).
\]
Recall  by Lemma \ref{lm:orth} that $\{\chi_{n}\}_{n\in\mathbb{Z}}$ is an orthogonal system in $L^{2}[-1,0]$.
Then by \eqref{eq:lim-N-m} and \cite[Theorem 5, p.16]{Olevskii},
there exists a function $\phi_{\star}\in \mathcal{S}_{\mathbb{C}}$
and a positive measure set $\mathcal{U}\subset (-1,0)$ such that
\begin{eqnarray}\label{lim-infty}
\lim_{\ell\to +\infty}|\sum_{k=\mathcal{N}_{\ell_{0}}+1}^{\mathcal{N}_{\ell}}
   \left(\int^{0}_{-1} \overline{\chi}_{\ell_{k}}(s) \phi_{\star}(s) d s \right)\chi_{\ell_{k}}(\theta)|=+\infty,
\ \ \ \ \ \forall\,\theta\in\mathcal{U}.
\end{eqnarray}
Define $\phi_{*}$ in $\mathcal{S}_{\mathbb{C}}$ by
\[
\phi_{*}(\theta)=e^{(\ln |c|+{\bf i}\arg (-c))\theta}\phi_{\star}(\theta),
\ \ \ \ \ \theta\in [-1,0].
\]
Then using  \eqref{eq:Tm-Tm0} and \eqref{lim-infty} yields that for each $\theta\in\mathcal{U}$,
\begin{eqnarray*}
&& \lim_{\ell\to +\infty}|(\mathcal{T}_{\ell}-\mathcal{T}_{\ell_{0}})\phi_{*}(\theta)|\\
&&\ \ \
   =\lim_{\ell\to +\infty} |\sum_{k=\mathcal{N}_{\ell_{0}}+1}^{\mathcal{N}_{\ell}}
   \left(\int^{0}_{-1} \overline{\chi}_{\ell_{k}}(s) \phi_{*}(s) d s \right)\chi_{\ell_{k}}(\theta)   e^{(\ln |c|)\theta}|
   =+\infty.
\end{eqnarray*}
This implies the first statement.
The second one follows from the Banach-Steinhaus Theorem \cite[p.98]{Rudin-87}.
Therefore, the proof is now completed.
\end{proof}

This lemma plays an important role in the subsequent proof.
It will be used to construct a function $\phi_{\star}$ in $C[-,0]$ such that
$\sup_{m\in\mathbb{N}}\|P_{m}^{\pm}\phi_{\star}\|= +\infty$.
This shows that the partial sums of the dichotomy projections acting on this function are divergent.


\subsection{Proof of Theorem \ref{thm:2}}\

In this subsection,
we further present some properties of the eigenvalues of $\mathcal{A}$,
i.e., the roots of Eq. \eqref{eq:character}.
They will be used to describe the changes of  the dimensions associated with sequential dichotomies.

For each $\delta>0$,
let  an integer-valued function $\Theta: (0,+\infty)\to \mathbb{N}$ be defined by
\begin{eqnarray}\label{df:Theta}
\Theta(\delta):=\#\left\{\lambda\in \sigma(\mathcal{A}):|{\rm Re}\lambda-\ln|c||> \delta\right\}.
\end{eqnarray}
By Lemmas \ref{lm:eigen-distribution}, \ref{lm:conf-c-positve} and \ref{lm:conf-c-negative},
$\Theta$ has the following properties:
\begin{enumerate}
\item[$\bullet$] $\Theta$ is right-continuous;
\vskip 3pt

\item[$\bullet$] $\Theta$ is nonincreasing and there exists a constant $\lambda_{*}$ such that  $\Theta(\delta)=0$ for all $\delta\geq \lambda_{*}$;
\vskip 3pt

\item[$\bullet$] $\Theta$ takes finitely many integer values if $|a+\ln |c||= |\ln |c|+\frac{b}{c}|$,
and $\Theta$ takes countably many integer values if $|a+\ln |c||\neq |\ln |c|+\frac{b}{c}|$.

\end{enumerate}
Here we  focus on the case $|a+\ln |c||\neq |\ln |c|+\frac{b}{c}|$.
Let  all values of the integer-valued function  $\Theta$  be denoted by $\mathcal{K}_{m}$, $m\in\mathbb{N}$, which satisifes
\begin{eqnarray}\label{eq:df-K-m}
0=\mathcal{K}_{0}<\mathcal{K}_{1}<\mathcal{K}_{2}<\cdots<\mathcal{K}_{m}<\cdots,
\end{eqnarray}
and $\lim_{m\to +\infty}\mathcal{K}_{m}= +\infty$.
In order to describe the decay rate of $\mathcal{K}_{m}$,
we give the next lemma.

\begin{lemma}\label{lm:eigen-strip}
For each real constant $x\neq \ln |c|$,
there is at most one pair of conjugate complex roots $\lambda$ and $\bar\lambda$  of Eq. \eqref{eq:character}
such that $\lambda$ lies on the line ${\rm Re}\lambda=x$ and ${\rm Im}\lambda\neq 0$.
\end{lemma}
\begin{proof}
Suppose that there exists a real constant $x_{0}$ with $x_{0}\neq \ln |c|$ such that
Eq. \eqref{eq:character} has two  pairs of conjugate complex roots $x_{0}\pm {\bf i}y_{1}$ and $x_{0}\pm {\bf i}y_{2}$,
where $y_{1}>0, y_{2}>0$ with $y_{1}\neq y_{2}$.

Let $\lambda:=x+{\bf i}y$. Then  Eq. \eqref{eq:character} is equivalent to
\begin{eqnarray}
y(ce^{-x}\sin y)
   \!\!\!&=&\!\!\! -(a+be^{-x}\cos y)-x(1+ce^{-x}\cos y),\label{eq:x-1}\\
x(ce^{-x}\sin y)+be^{-x}\sin y
   \!\!\!&=&\!\!\! y(1+ce^{-x}\cos y).\label{eq:y-1}
\end{eqnarray}
Since  $x_{0}\neq \ln |c|$, we have $1-ce^{-x_{0}}\neq 0$.
This together with \eqref{eq:y-1} yields $y_{j}\notin \{2k\pi+\pi:\, k\in \mathbb{Z}\}$.
In fact, substituting $y=2k\pi+\pi$ in \eqref{eq:y-1},
we see that   the left-hand side of \eqref{eq:y-1} is zero and the right-hand side is nonzero.

Multiplying \eqref{eq:x-1} by \eqref{eq:y-1} yields
\begin{eqnarray*}
y_{j}\left\{((a + 2x_{0})c+b)e^{-x_{0}}\cos y_{j}+c(b+cx_{0})e^{-2x_{0}}+a+x_{0}\right\}=0.
\end{eqnarray*}
Since $y_{j}\neq 0$ for $j=1,2$, we have
\begin{eqnarray}\label{eq:tan-varrho}
(\tan(\frac{y_{j}}{2}))^{2}
=\frac{((a + 2x_{0})c + b)e^{-x_{0}}-c(cx_{0}+b)e^{-2x_{0}} - a - x_{0}}
       {((a + 2x_{0})c + b)e^{-x_{0}} + c(cx_{0} + b)e^{-2x_{0}} + a + x_{0}}=:\varrho(x_{0}).
\end{eqnarray}
Then  $y_{j}$ ($j=1,2$) satisfy
either $\tan(\frac{y_{j}}{2})=\sqrt{\varrho(x_{0})}$ or $\tan(\frac{y_{j}}{2})=-\sqrt{\varrho(x_{0})}$.

If $\tan(\frac{y_{1}}{2})=\tan(\frac{y_{2}}{2})=\sqrt{\varrho(x_{0})}$,
then
\begin{eqnarray*}
\sin y_{1}=\sin y_{2}=\frac{\sqrt{\varrho(x_{0})}}{1+\varrho(x_{0})},
\ \ \
\cos y_{1}=\cos y_{2}=\frac{1-\varrho(x_{0})}{1+\varrho(x_{0})}.
\end{eqnarray*}
By \eqref{eq:x-1} and \eqref{eq:y-1},
\begin{eqnarray}\label{eq:contract-1}
\begin{aligned}
(y_1-y_2) ce^{-x_{0}}\sin y_1  =& 0,\\
(y_1-y_2) (1+ce^{-x_{0}}\cos y_{1})=&0.
\end{aligned}
\end{eqnarray}
Since $x_{0}\neq \ln |c|$,  we can compute
\begin{eqnarray*}
(ce^{-x_{0}}\sin y_{1})^{2}+(1+ce^{-x_{0}}\cos y_{1})^{2}=c^{2}e^{-2x_{0}}+2ce^{-x_{0}}\cos y_{1}+1\geq (1-|c|e^{-x_{0}})^{2}>0.
\end{eqnarray*}
This together with \eqref{eq:contract-1} yields $y_{1}=y_{2}$, which is a contradiction.

If $\tan(\frac{y_{1}}{2})=-\tan(\frac{y_{2}}{2})=\sqrt{\varrho(x_{0})}$,
then $\tan(\frac{y_{1}}{2})=\tan(\frac{-y_{2}}{2})=\sqrt{\varrho(x_{0})}$.
Note that $(x,y)=(x_{0},y_{1})$ and $(x,y)=(x_{0},-y_{2})$ solve \eqref{eq:x-1} and \eqref{eq:y-1}, respectively.
Similarly to the preceding discussion, we have $y_{1}=-y_{2}$, which is a contradiction to the fact $y_{1}>0>-y_{2}$.

Similarly, we can obtain contradictions in the cases that
$-\tan(\frac{y_{1}}{2})=\tan(\frac{y_{2}}{2})=\sqrt{\varrho(x_{0})}$
and  $\tan(\frac{y_{1}}{2})=\tan(\frac{y_{2}}{2})=-\sqrt{\varrho(x_{0})}$.
Thus this lemma holds.
\end{proof}

We next give the following lemma to indicate the decay rate of $\mathcal{K}_{m}$.

\begin{lemma}\label{lm:lim-K-n}
If $|a+\ln |c||\neq |\ln |c|+\frac{b}{c}|$,
then
\[
\lim_{m\to +\infty}\frac{\mathcal{K}_{m+1}}{\mathcal{K}_{m}}=1.
\]
\end{lemma}
\begin{proof}
By Lemmas \ref{lm:eigen-distribution} and \ref{lm:eigen-strip},
we have
\begin{eqnarray}\label{contr-gamma}
\mathcal{K}_{m}+1\leq \mathcal{K}_{m+1}\leq \mathcal{K}_{m}+7.
\end{eqnarray}
Recall that $\lim_{m\to +\infty}\mathcal{K}_{m}=+\infty$. Then
\[
\lim_{m\to +\infty}\frac{\mathcal{K}_{m}+1}{\mathcal{K}_{m}}=\lim_{m\to +\infty}\frac{\mathcal{K}_{m}+7}{\mathcal{K}_{m}}=1,
\]
which together with \eqref{contr-gamma} yields this lemma.
\end{proof}

Finally, we prove the second  theorem on the sequential dichotomies of Eq. \eqref{eq:NE}.

{\bf Proof of Theorem \ref{thm:2}.}
If $c\ne 0$ and (\ref{abc}) is true,
then by Lemmas \ref{lm:conf-c-positve} and \ref{lm:conf-c-negative}
the dichotomous resolvent set $\Omega$ (see \eqref{df:Sigma-Omega}) consists of countably infinite many connected intervals of $\lambda$
satisfying that ${\rm Re}\lambda<\ln |c|$ or ${\rm Re}\lambda>\ln |c|$.
We only prove the first case and the other ones can be discussed similarly.

Without loss of genrality,
let $\mathcal{E}(\beta_{m}, \alpha_{m}; K_{m})$, $m \in\mathbb{N}$,
denote all different dichotomies of Eq. \eqref{eq:NE} ordered according to
\[
\beta_{0}<\alpha_{0}<\beta_{1}<\alpha_{1}<\cdots<\beta_{m}<\alpha_{m}<\cdots.
\]
Then by  (i) and (iv) of Lemmas \ref{lm:conf-c-positve} and \ref{lm:conf-c-negative},
each stable projection $P_{m}^{-}$ is the sum of finitely many spectral projections and $P_{0}^{-}=0$.
In the following,
we give the explicit expressions of projections $P_{m}^{-}$ for $m \ge 1$.
Let
\[
\widetilde{\Omega}_{m}:=\{\lambda\in \sigma(\mathcal{A})\,:\, {\rm Re}\lambda< \beta_{m}\},
\ \ \ \ \
\widetilde{\mathcal{K}}_{m}:=\#\widetilde{\Omega}_{m},
\ \ \ \ \ m \geq 1.
\]
Recall that $\Theta$ is defined by \eqref{df:Theta}. By Lemma \ref{lm:App-1},
there is at most one eigenvalue $\lambda$ of $\mathcal{A}$ such that
${\rm Re}\lambda>\ln |c|$.
Then
\begin{eqnarray}\label{est:K-mathcalK}
\widetilde{\mathcal{K}}_{m}\leq \mathcal{K}_{m}\leq \widetilde{\mathcal{K}}_{m}+1,
\end{eqnarray}
where $\mathcal{K}_{m}$ are defined by \eqref{eq:df-K-m}.
Hence we can write  $\widetilde{\Omega}_{m}$ as
\[
\widetilde{\Omega}_{m}=\{\xi_{1},\xi_{2},...,\xi_{\widetilde{\mathcal{K}}_{m}}\},
\]
and $P_{m}^{-}\phi$ for $\phi\in\mathcal{C}_{\mathbb{C}}$ as
\[
P_{m}^{-}\phi=\sum_{i=1}^{\widetilde{\mathcal{K}}_{m}}P_{\xi_{i}}\phi,
\ \ \ \ \  m \geq 1,
\]
where each $P_{\xi_{i}}$  is the spectral projection of $\xi_{i}$.

To prove (i),
it suffices to prove $\sup_{m\geq 1}\|P_{m}^{-}\|=+\infty$ because $\|P_{m}^{-}+P_{m}^{+}\|=1$ for all $m\geq 1$.
We want to apply Lemma \ref{lm:Fseries-nonuniform} to complete it.
So we first construct a family of sets $\Omega_{\ell}$ such that  the properties in ({\bf P}) and \eqref{eq:lim-N-m} hold.
Let $\varepsilon$, $n_{\varepsilon}$ and $B_{n}(\varepsilon)$ be defined as in (ii) of Lemma  \ref{lm:eigen-distribution}.
By (ii) and (iv) of Lemma  \ref{lm:eigen-distribution}
there exists an integer $m_{0}$ such that for all $m\geq m_{0}$,
\[
(\widetilde{\Omega}_{m}\setminus \widetilde{\Omega}_{m_{0}})
  =\{\xi_{\widetilde{\mathcal{K}}_{m_{0}}+1},\xi_{\widetilde{\mathcal{K}}_{m_{0}}+2},...,\xi_{\widetilde{\mathcal{K}}_{m}}\} \subset \cup_{|n|>n_{\varepsilon}} B_{n}(\varepsilon).
\]
Let $\tilde{\xi}_{\widetilde{\mathcal{K}}_{m_{0}}+i}$ ($1\leq i\leq \widetilde{\mathcal{K}}_{m}-\widetilde{\mathcal{K}}_{m_{0}}$)
denote
the centers of those $B_{n}(\varepsilon)$s containing $\xi_{\widetilde{\mathcal{K}}_{m_{0}}+i}$,
and
\[
\Omega_{\ell}=\{\tilde{\xi}_{\widetilde{\mathcal{K}}_{m_{0}}+1},\tilde{\xi}_{\widetilde{\mathcal{K}}_{m_{0}}+2},...,\tilde{\xi}_{\widetilde{\mathcal{K}}_{m_{0}+\ell}}\},
\ \ \ \ \ \ell \geq 1.
\]
Since the eigenvalues of $\mathcal{B}$ are the centers of $B_{n}(\varepsilon)$\,s,
we see that $\Omega_{\ell}\subset \sigma(\mathcal{B})$ and $\{\Omega_{\ell}\}_{\ell=1}^{\infty}$ satisfies ({\bf P}).
We  still use $\mathcal{N}_{\ell}$ to denote $\#\Omega_{\ell}$.
By the fact that
$
\mathcal{N}_{\ell}=\widetilde{\mathcal{K}}_{m_{0}+\ell}-\widetilde{\mathcal{K}}_{m_{0}},
$
\eqref{est:K-mathcalK} and Lemma \ref{lm:lim-K-n},
\[
\lim_{\ell\to +\infty}\frac{\mathcal{N}_{\ell+1}}{\mathcal{N}_{\ell}}
=\lim_{\ell\to +\infty}\frac{\widetilde{\mathcal{K}}_{m_{0}+\ell+1}-\widetilde{\mathcal{K}}_{m_{0}}}{\widetilde{\mathcal{K}}_{m_{0}+\ell}-\widetilde{\mathcal{K}}_{m_{0}}}
=\lim_{\ell\to +\infty}\frac{\mathcal{K}_{\ell+1}}{\mathcal{K}_{\ell}}=1.
\]
Thus we get a family of desirable  sets $\Omega_{\ell}$.

Recall that $\mathcal{T}_{\ell}$ is defined by  \eqref{df:T-ell}.
By Lemma \ref{lm:Fseries-nonuniform},
there exists a function $\phi_{*}\in \mathcal{S}_{\mathbb{C}}$ and  $\mathcal{U}\subset (-1,0)$ such that
\[
\lim_{\ell\to +\infty} |\mathcal{T}_{\ell}\phi_{*}(\theta)| =+\infty,
\ \ \ \ \ \forall\,\theta\in\mathcal{U}.
\]
It follows that there is a point $\theta_{\star}\in \mathcal{U}$ such that
either $\phi_{\star}:={\rm Re}\phi_{*}/\|{\rm Re}\phi_*\|_{\infty}$ or $\phi_{\star}:={\rm Im}\phi_*/\|{\rm Im}\phi_{*}\|_{\infty}$
satisfies
\begin{eqnarray}\label{infty-est}
\lim_{\ell\to +\infty} |\mathcal{T}_{\ell}\phi_{\star}(\theta_{\star})| =+\infty.
\end{eqnarray}
Furthermore, this $\phi_{\star}$ is in the set $\mathcal{S}_{\mathbb{C}}\cap C[-1,0]$ and $\|\phi_{\star}\|_{\infty}=1$.
For each $m\geq m_{0}+1$, we write $P_{m}^{-}\phi_{\star}$ as
\begin{eqnarray}\label{eq:prj-split}
P_{m}^{-}\phi_{\star}=
           P_{m_{0}}^{-}\phi_{\star}
          +\sum_{i=\widetilde{\mathcal{K}}_{m_{0}}+1}^{\widetilde{\mathcal{K}}_{m}}(P_{\xi_{i}}\phi_{\star}-Q_{\tilde{\xi}_{i}}\phi_{\star})
          +\mathcal{T}_{\ell}\phi_{\star},
\end{eqnarray}
where  $Q_{\tilde{\xi}_{i}}$s
are defined by \eqref{df:Q-zn} and
\[
\mathcal{T}_{\ell}\phi_{\star}=\sum_{z\in\Omega_{\ell}}Q_{z}\phi_{\star}
                          =\sum_{i=\widetilde{\mathcal{K}}_{m_{0}}+1}^{\widetilde{\mathcal{K}}_{m_{0}+\ell}}Q_{\tilde{\xi}_{i}}\phi_{\star},
\ \ \ \ \
\ell=m-m_{0}.
\]
By (iv) of Lemma \ref{lm:A-spectrum} and the argument below \eqref{eq:proj},
each term in \eqref{eq:prj-split} belongs to $C[-1,0]$.
Note that
\[
P_{m_{0}}^{-}\phi_{\star}=\sum_{i=1}^{\widetilde{\mathcal{K}}_{m_{0}}}P_{\xi_{i}}\phi_{\star}
\]
is independent of $m$.
This together with  Lemma \ref{pert-prjs} yields that there exists a constant $C_{\#}>0$, independent of $m$, such that
\begin{eqnarray}\label{est:sum-prj-bdd}
\|P_{m_{0}}^{-}\phi_{\star}\|_{\infty}
+\|\sum_{i=\widetilde{\mathcal{K}}_{m_{0}}+1}^{\widetilde{\mathcal{K}}_{m}}(P_{\xi_{i}}\phi_{\star}-Q_{\tilde{\xi}_{i}}\phi_{\star})\|_{\infty}
\leq C_{\#}.
\end{eqnarray}
Applying \eqref{infty-est}  and \eqref{est:sum-prj-bdd} to \eqref{eq:prj-split}, we have that
\[
\lim_{m\to +\infty}|P_{m}^{-}\phi_{\star}(\theta_{\star})|=+\infty.
\]
Hence $\sup_{m\geq 1}\|P_{m}^{-}\|=+\infty$ by the Banach-Steinhaus Theorem (see \cite[p.98]{Rudin-87}).
By the definition of exponential dichotomies,
we have that $K_{m}\geq \|P_{m}^{-}\|$. This implies
\[
\sup_{m\geq 1}K_{m}\geq \sup_{m\geq 1}\|P_{m}^{-}\|=+\infty.
\]
Therefore the proof is now completed.
\hfill $\Box$

The present work finally settles two basic questions on Eq. \eqref{eq:NE},
i.e., the sequential uniformity and the convergence problem for its projections series.
It is worth mentioning that
neutral equations as semigroups on various Banach spaces, except for $C[-1,0]$ considered in this paper,
were once investigated in many references (see \cite{BDC-18,Hadd-08} for instance).
Using the orthonormal systems formed by the eigenfunctions of Eq. \eqref{eq:NE-auxi},
our argument is also applicable to discuss Eq. \eqref{eq:NE} on another  Banach spaces
such as $L^{2}[-1,0]$.


\bibliographystyle{amsplain}

\begin{thebibliography}{99}


\bibitem{Banks-Manitus} H. T. Banks, A. Manitius, Projection seris for retarded functional differential equations with applications
to optimal control problems,
 {\it J. Differential Equations} {\bf 18} (1975), 296--332.

\bibitem{BDC-18} L. Barreira, D. Dragi\v{c}evi\'c,  C. Valls,
    {\it Admissibility and Hyperbolicity},  Springer Briefs in Mathematics, Springer, Cham, 2018.

\bibitem{Bellemmaan-Cooke-63} R. Bellman, K. Cooke, {\it Differential--Difference Equations}, Academic Press, New York, 1963.


\bibitem{Chen-Shen-20} S. Chen, J. Shen,
Large spectral gap induced by small delay and its application to reduction,
{\it Discrete Contin. Dyn. Syst.} {\bf 40} (2020), 4533--4564.

\bibitem{Coppel-78} W. Coppel,
{\it Dichotomies in Stability Theory},
Lecture Notes in Math., Vol. {\bf 629}, Springer, New York, 1978.

\bibitem{Hadd-08} S. Hadd, Singular functional differential equations of neutral type in Banach spaces,
{\it J. Functional. Anal.} {\bf 254} (2008), 2069--2091.


\bibitem{JKHale-Verduyn} J. K. Hale, S. M. Verduyn Lunel,
   {\it Introduction to Functional Differential Equations},
   Appl. Math. Sci., Vol. {\bf 99},  Springer, New York, 1993.

\bibitem{Hale-Zhang-04} J. Hale, W. Zhang,
On uniformity of exponential dichotomies for delay equations,
{\it J. Differential Equations} {\bf 204} (2004), 1--4.


\bibitem{Henry-81} D. Henry,
   {\it Geometric Theory of Semilinear Parabolic Equations},
   Lecture Notes in Math., Vol. {\bf 840}, Springer, Berlin, 1981.


\bibitem{Junca-Lombard-14} S. Junca,  B. Lombard,
Stability of a critical nonlinear neutral delay differential equation,
{\it J. Differential Equations} {\bf 256} (2014), 2368--2391.

\bibitem{Kato-95} T. Kato,
{\it Perturbation Theory for Linear Operators},
Classics in Mathematics, Vol. {\bf 132}, Springer, Berlin, 1995.



\bibitem{Olevskii} A. M. Olevskii,
{\it Fourier Series with Respect to General Orthonormal Systems},  Ergebnisse der Mathematik und ihrer Grenzgebiete, Vol. {\bf 86}, Springer, Berlin, 1975.


\bibitem{Pazy-83} A. Pazy,
{\it Semigroups of Linear Operators and Applications to Partial Differential Equations},
 Appl. Math. Sci., Vol. {\bf 44}. Springer, New York, 1983.


\bibitem{Perron-30} O. Perron, Die stabilit\"atsfrage bei differentialgleichungen,
 {\it Math. Z.} {\bf 32} (1930), 703--728.

\bibitem{Rudin-87} W. Rudin,
{\it Real and Complex Analysis}, 3rd ed., McGraw-Hill,
 New York, 1987.

\bibitem{Verduyn-01} S. M. Verduyn Lunel,
Spectral theory for delay equations,
In: {\it Systems, Approximation, Singular Integral Operators, and Related Topics},
Eds. A. A. Borichev, N. K. Nikolski,
Oper. Theory Adv. Appl., Vol. {\bf 129}, Birkh\"auser, Basel, 2001, pp. 465--508.


\bibitem{Zhang-93}  W. Zhang,
  Generalized exponential dichotomies and invariant manifolds for differential equations,
{\it Adv. in Math. (China)} {\bf 22} (1993), 1--45.

\bibitem{Zhou-Zhang-16} L. Zhou, W. Zhang,
Admissibility and roughness of nonuniform exponential dichotomies for difference equations,
{\it J. Functional. Anal.} {\bf 271} (2016), 1087--1129.



\end{thebibliography}

\end{document}